\documentclass[aps,superscriptaddress%
 preprint,showkeys,
 amsmath,amssymb,
 aps, physrev, float
]{revtex4-2}
\usepackage{float}
\usepackage{booktabs}
\usepackage[left=2.5cm,right=2.5cm,top=2cm,bottom=2cm]{geometry}
\usepackage[normalem]{ulem}
\usepackage{amssymb,amsthm,amsmath,amsfonts}
\usepackage{graphicx}
\usepackage{dcolumn}
\usepackage{bm}
\newtheorem{theorem}{Theorem}
\newtheorem{algorithm}{Algorithm}
 \newtheorem{lemma}{Lemma}

 \newtheorem{definition}{Definition} 
\begin{document}

\preprint{APS/123-QED}

\title{\textbf{Exploring topological features of a complex network of prime factors and composite numbers} 
}%

\author{Carlos Daniel Tamayo-Castro}
 \email{cdtamayoc@ciencias.unam.mx}
 \affiliation{
Facultad de Ciencias. Universidad Nacional Aut\'onoma de M\'exico. Ciudad de M\'exico, México. \\ESIME-Zacatenco-Instituto Polit\'ecnico Nacional. 07338, Ciudad M\'exico, M\'exico.
}%

\author{Aldo Ramirez-Arellano}%
 \email{aramirezar@ipn.mx}
\affiliation{UPIICSA-Instituto Polit\'ecnico Nacional. 07338, Ciudad M\'exico, M\'exico.
}

\author{Juan Bory-Reyes}
  \email{juanboryreyes@yahoo.com}
\affiliation{
ESIME-Zacatenco-Instituto Polit\'ecnico Nacional. 07338, Ciudad M\'exico, M\'exico.
}%

\date{\today}

\begin{abstract}

Prime numbers are mathematical concepts of particular interest because of their relative rarity among integers, as well as the relative unpredictability of their occurrence and density along the integer sequence. They constitute the building blocks of all other composite numbers, giving rise to the interesting and important concept of prime factorization, a process guaranteed by the fundamental theorem of arithmetic. Among the several works that have been reported about prime numbers and their interrelationship with composite numbers, some have considered derived graphs and/or networks.
\noindent 
 
In this paper, we derive some topological features of a bipartite complex network of prime factors and composite numbers. More specifically, we prove mathematically that the network's diameter is exactly five regardless of the number of nodes. In particular, we show that certain nontrivial sub-networks have minimal diameter equal to four, which is, again, independent of the number of nodes as long as the network is large enough. The last shows that this network is of small-world type. We also characterize exactly the number of boundary nodes in the network and present a couple of algorithms to identify the most important, relevant, influential, or central nodes of the network. Finally, we focus on the computation of the information dimensions on the basis of Shannon and Tsallis entropies of the studied complex network. The results show both decreasing as the number of nodes increases. Furthermore, both entropies indicate that box coverings at the diameters \(\varepsilon=[3,4]\), group nodes into roughly the same number of boxes, independent of the number of nodes in the network.

\end{abstract}

\maketitle


\section{Introduction}
Prime numbers have fascinated humans for a long period of time. Deriving from the Latin term $\textit{primus}$, meaning first, prime numbers occupy a special place along the sequence of integer values, especially because they constitute primary factors from which other (composite) subsequent numbers can be obtained by factor multiplication \cite{CranPom,Ingham,IreRo1990,Na2006,Stillwell,Wells}. More specifically, prime numbers are integers divisible only by 1 and themselves. For the sake of clarity, only positive integer numbers are considered henceforth in this work, and the number 1 will be excluded as a prime number.

Prime and composite numbers differ based on the number of factors they have. A prime number has only two factors, and a composite number has more than two factors. A factor is a value that can divide a number or an expression evenly. 

The decomposition of an integer into its prime factors, which is unique, constitutes an especially interesting operation in number theory, called prime factor decomposition. By providing a basic “strategy” for building integer numbers from prime components, prime factor decomposition also paves the way to establishing relationships between integer numbers that take into account common prime factors and divisibility.

The mystery of primes is not a mere conundrum of pure mathematics. Unexpected connections can be discovered between primes and different topics in applied research areas such as 
physics, biology, chemistry, communication, acoustics, electronics, cryptography, computing, feedback and control in engineering \cite{Ya}.  

Information theory has also been applied to the study of prime numbers, as Kontoyiannis \cite{Kont2008} demonstrated that basic information-theoretic arguments and elementary probability combined can prove Chebyshev's classical asymptotic result for the sum of logarithmic prime ratios, and Mayer et al. \cite{MaMinSaMo2024} established fundamental information-theoretic inequalities for positive integers and ideals in rings of algebraic integers. These works are founded on the definition of entropy, establish upper and lower bounds, and explore how this measure reflects the structural complexity, or "disorder," of integers \cite{MinPo2011}. Expanding this framework, the authors of \cite{MinSa2023} introduce a novel way to calculate the entropy of a non-zero ideal within the ring of integers of an algebraic number field, utilizing its unique factorization into prime ideals and their ramification indices.

One of the most promising approaches to solving the enigmas of number theory is the advances in complex network theory; see \cite{Azaria,ChanDas,GaABo2014,RaAm2020,SheBhaAm,ZBHC}. At the same time, the deep mathematical apparatus of number theory has provided a range of tools for the study of complex networks. These tools make it possible to construct networks of primes and composite numbers as extreme cases, which push the boundaries of the definitions, metrics, and algorithms of complex network theory.

The order relation implicit in the natural numbers allows us to consider a bipartite network representation of natural numbers as a growing system. In the growing process, natural numbers join the network sequentially and try to connect to already existing primes. Those new numbers that succeed in this process are said to be composite; otherwise, they become prime numbers. 

In this paper, we derive some topological features of a bipartite complex network of prime factors and composite numbers. More specifically, we prove mathematically that the network's diameter is exactly five. In particular, we show that certain nontrivial sub-networks have a minimum diameter of four, and we show that this network possesses the small-world property. We also characterize exactly the number of boundary nodes in the network and present a couple of algorithms to identify the most important, relevant, influential, or central nodes.

\section{Background}
In this section, we provide a brief exposition of foundational concepts according to our purpose; see \cite{CosRoTra2007,CranPom, E2012, Ingham,IreRo1990,Na2006,Stillwell,Wells} for more details.
\subsection{Complex Networks}
Most of the complexity of networks is encoded into the intricate topology of the interactions among their components and into the layout of the intensities associated with such interactions.

Each complex network (or class of networks) presents specific topological features which characterize its connectivity and highly influence the dynamics of processes executed on the network. 
\begin{definition}(Complex Network)\\
A complex network is a large, simple graph whose structure exhibits significant heterogeneity—for example, a highly non-uniform degree distribution—and non-trivial global properties (clustering, communities, short paths) that do not conform to the typical behavior of regular graphs or the classical random model.	
\end{definition}
Among the vast variety of complex networks, bipartite networks form an important class. The networks studied here are a particular case of these.
\begin{definition}(Bipartite Network)
In this type of network, the set of nodes $\mathcal{N}$ can be written as the disjoint union of two sets $\mathcal{N}_{1}$ and $\mathcal{N}_{2}$ such that there are no edges in the network of the form $\{i,j\}\in \mathcal{E}$ where $i\in \mathcal {N}_{1}$ and $j\in \mathcal{N}_{2}$.		
\end{definition}

\begin{definition}(Network diameter)
Let $G = (V, E)$ be a connected complex network. The diameter of $G$, denoted $\mathrm{diam}(G)$, is defined as
\[
\mathrm{diam}(G) = \max_{u,v \in V} d(u,v),
\]
where $d(u,v)$ is the length of the shortest path between nodes $u$ and $v$.
\end{definition}

\begin{definition}(Degree of a Node)
In a complex network, the degree of a node $v \in V$, denoted $\deg(v)$, is defined as the number of edges incident to $v$. In undirected graphs, this
corresponds to the number of neighbors of $v$.
\end{definition}

\begin{definition}(Relevant Node)
In a complex network, relevant nodes are those vertices whose structural position has a significant impact on the network's global properties. 
\end{definition}

Identifying relevant nodes in complex networks is of great significance to the research of topological properties of the networks and has attracted extensive research attention  attention
and interest. In particular, the following algorithm has been efficiently used in \cite{barabasi2000,han2002,Ramirez2021}.
\begin{algorithm}\label{TeoAlgRelNodeClass}(The classical (disconnection) algorithm for identifying relevant nodes in a complex network)
\begin{enumerate}
	\item[(i)] The nodes are sorted from highest to lowest based on their degree.
	\item[(ii)] The first vertex is removed from the list, added to the list of relevant nodes, and a new network is formed with the remaining vertices.
	\item[(iii)] If the new network contains any connected sub-networks, steps (i) and (ii) are repeated; otherwise, the process ends. 
\end{enumerate}
\end{algorithm}

\begin{definition}(Boundary Nodes)
For the purposes of this paper, we will consider nodes with degree one to be boundary nodes.
\end{definition}

\subsection{Prime and Composite Numbers}
Prime numbers are the building blocks of the integers and a central object of study in number theory.
\begin{definition}(Prime Number)\\
A prime number $p$ is a natural number different from 1 that is divisible only by 1 and $p$. 	
\end{definition}

\begin{theorem}(Fundamental Theorem of Arithmetic)\\
Every integer $n > 1$ can be written in a unique way as
\[
n = p_1^{\alpha_1} p_2^{\alpha_2} \cdots p_k^{\alpha_k},
\]
where $p_1 < p_2 < \cdots < p_k$ are distinct prime numbers and $\alpha_1, \alpha_2, \dots, \alpha_k$ are positive integers. This factorization is unique up to the order of the factors.	
\end{theorem}

\begin{theorem}(Eratosthenes' Sieve Algorithm)\\
Let $N$ be an integer greater than $1$. To find all prime numbers less than or equal to $N$, the method consists of the following steps:

\begin{enumerate}
	\item[(i)] Write a list of integers from $2$ to $N$.
	\item[(ii)] Select the first number in the list that has not yet been checked; this number is a prime number.
	\item[(iii)] Mark (or conceptually eliminate) all multiples of this prime number that are greater than it.
	\item[(iv)] Repeat the previous steps, using the next unmarked number on the list as the new prime.
	\item[(v)] Stop the process when the square of the selected number exceeds $N$; the numbers that remain unchecked are prime.
\end{enumerate}		
\end{theorem}

In \cite{RosScho} we can find the following result, which plays an important role in the proofs of some of our results.
\begin{theorem}\label{TheoEstPrim}(Estimates of the amount of primes numbers)\\
Let $\pi(N)$ be the function that assigns to each value of $N$ the number of primes less than or equal to $N$. Then we have that
\begin{itemize}
	\item[(i)] \begin{eqnarray*}
		\dfrac{N}{\ln N} < \pi(N), & \text{for } 17 \leq N;
	\end{eqnarray*}
	
	\item[(ii)] \begin{eqnarray*}
		\pi(N) < 1.25506 \dfrac{N}{\ln N}, & \text{for } 1 < N.
	\end{eqnarray*}
\end{itemize}	
\end{theorem}
\section{Complex Network of Prime Factors and Composite Numbers}
Let us consider the bipartite network introduced in \cite{GaABo2014}, to be denoted by $R(N)$. Let all natural numbers from 2 to $N$ be considered, excluding prime numbers between $\dfrac{N}{2}$ and $N$. These numbers are the nodes of the network. From each prime number, we draw an edge to every composite number that has it in its prime factor decomposition. Each edge has a weight that is given by the multiplicity in the prime factor decomposition divided by the sum of all multiplicities in the decomposition.    
\subsection{Network Topological Structure}

\paragraph{Connectivity and Diameter}
\begin{theorem}\label{TheoConectNet}
	The network $R(N)$ is connected.
\end{theorem}
\begin{proof}
We should prove that any two nodes are connected through a path of edges. First, we will study the nodes corresponding to composite numbers and prime numbers less than $\sqrt{N}$. Later, we will see the connectivity of nodes corresponding to prime numbers between $\sqrt{N}$ and $\dfrac{N}{2}$. In our network, each node corresponds to an integer between 2 and $N$; then we can have two cases: The first case, the nodes are relatively prime numbers, or the second one, they are not. 
	
We will use the letters $q$ and $m$ to denote numbers corresponding to two different nodes.
	
	\textbf{Case 1} If $q$ and $m$ are not relative primes then there exists a prime number $p_{qm}$ which is in both prime factor decompositions of $q$ and $m$, hence
	\begin{equation*}
		q \rightarrow p_{qm}  \rightarrow m
	\end{equation*}
	so in this case $p$ and $q$ are connected through a path of two steps.
	
	\textbf{Case 2} If $q$ and $m$ are relative primes then we will denote $p_{q}$ and $p_{m}$ the lower factor in the prime factor decomposition of $q$ and $m$ then
	\begin{equation*}
		q \rightarrow p_{q} \rightarrow p_{q}p_{m} \rightarrow p_{m} \rightarrow m
	\end{equation*}
	Here we should note that, because $p_{q}$ and $p_{m}$ are the lower factors in the prime factor decomposition of $q$ and $m$ and they cannot be prime numbers between $\sqrt{N}$ and $\dfrac{N}{2}$, then the number $p_{q}p_{m} \leq \sqrt{N}\sqrt{N} = N$ is a composite number in the network.
	Thus, in this case $p$ and $m$ are connected through a path of at most four steps.
	
	We have just proved that the sub-network $R^{'}(N)$ consisting of all prime numbers less than or equal to $\sqrt{N}$ and all composite numbers less than or equal to $N$ forms a connected network. In order to prove the connectivity of our original network $R(N)$, we only need to connect all prime numbers between $\sqrt{N}$ and $\dfrac{N}{2}$ to a node of $R^{'}(N)$. Let $p^{*}$ be a prime number between $\sqrt{N}$ and $\dfrac{N}{2}$, thus the compound $2p^{*} \leq N$ belongs to the network $R^{'}(N)$. Thus
	\begin{equation*}
		\cdots \rightarrow 2p^{*} \rightarrow p^{*}, 
	\end{equation*}
	which completes the proof.
\end{proof}
The following lemma plays a significant role in the proof of the next theorem. 
\begin{lemma}\label{LemPriN3N2}
	There exists at least one prime number in the interval $\left( \dfrac{N}{3}, \dfrac{N}{2}\right] $  for any value of $N\geq 4$ with $N\neq 9$ and $N\neq 21$.  
\end{lemma}
\begin{proof}
The number of prime numbers between $\dfrac{N}{3}$ and $\dfrac{N}{2}$ is $\pi(\frac{N}{2}) - \pi(\frac{N}{3})$. From Theorem \ref{TheoEstPrim} we get that
\begin{align*}
\dfrac{\frac{N}{2}}{\ln\frac{N}{2}} - 1.25506 \dfrac{\frac{N}{3}}{\ln\frac{N}{3}} < \pi(\frac{N}{2}) - \pi(\frac{N}{3}).
\end{align*}
It can be shown through some tedious computations that the function
\begin{align*}
f(x) =	\dfrac{\frac{x}{2}}{\ln\frac{x}{2}} - 1.25506 \dfrac{\frac{x}{3}}{\ln\frac{x}{3}} 
\end{align*}
is an increasing function for $x \geq 24$. Therefore for any $N \geq 116$ we have,
\begin{align*}
1 < f(116) \leq f(N) =	\dfrac{\frac{N}{2}}{\ln\frac{N}{2}} - 1.25506 \dfrac{\frac{N}{3}}{\ln\frac{N}{3}} < \pi(\frac{N}{2}) - \pi(\frac{N}{3}).
\end{align*}
For values of $N$ between 4 and 115, this can be checked manually with the aid of a computer program. That completes the proof.
\end{proof}

We are now in a position to prove the subsequent result. 
\begin{theorem}\label{TheoDiam5}
	The diameter of $R(N)$ is 5 for any value of $N \in \mathbb{N}$.
\end{theorem}
\begin{proof}
	Our proof starts with the observation that there is a prime number between $\dfrac{N}{3}$ and $\dfrac{N}{2}$, and that any two nodes of $R(N)$ can be linked through a path of at most five steps. Therefore, it only remains to prove that there exists a pair of nodes such that the shortest path between them is exactly five steps.
	
	Indeed, from Lemma \ref{LemPriN3N2} we have that there exists at least one prime number $\widetilde{p}$ between $\dfrac{N}{3}$ and $\dfrac{N}{2}$ for any value of $N$ big enough. Due to the fact that $ \dfrac{N}{3} < \widetilde{p} \leq \dfrac{N}{2}$, then $p\widetilde{p} \leq N$ implies that $p = 2$. Hence $\widetilde{p}$ is only connected to the node $2\widetilde{p}$, and this compound number is only connected to $\widetilde{p}$ and 2. Therefore, we have a unique path of two steps connecting $\widetilde{p}$ and 2,
	\begin{equation*}
		2 \rightarrow 2\widetilde{p} \rightarrow \widetilde{p} 
	\end{equation*}   	
	We are reduced to proving that there exists a number such that the minimal path to the node 2 is exactly three steps. In order to do that, we are going to pick an odd composite number $m$. A composite number is only connected to its prime factors. However, because $m$ is an odd number, $m$ is not connected to node 2. We will denote $p_{m}$ an arbitrary prime factor of $m$.\\
	Again, we only need to show that the minimal path between $p_{m}$ and node 2 has two steps. Because a prime number is connected only to composite numbers, there does not exist a one-step path from $p_{m}$ to the node 2. Thus, the shortest path should contain at least two steps. Indeed, the composite number $2p_{m}$ is connected to both nodes and is less than $N$. Hence, the path
	\begin{equation*}
		p_{m} \rightarrow 2p_{m} \rightarrow 2,
	\end{equation*} 
	is a minimal path between $p_{m}$ and 2. Consequently, the path 
	\begin{equation*}
		m \rightarrow	p_{m} \rightarrow 2p_{m} \rightarrow 2 \rightarrow 2\widetilde{p} \rightarrow \widetilde{p},
	\end{equation*}
	is a minimal path between an odd compound number $m$ and a prime number $\widetilde{p}$ between $\dfrac{N}{3}$ and $\dfrac{N}{2}$, and the proof is completed.
\end{proof}
We should note that this is an example of a network with an arbitrary number of nodes where the diameter is always five. To the best of the authors' knowledge, there is no previous example of a complex network with a constant diameter regardless of the increase in the number of nodes.    

\paragraph{Relevant and Boundary Nodes}

\begin{theorem}(Network Boundary Nodes)
The boundary nodes of the network are the nodes corresponding to numbers of the form $p_{i}^{\alpha_{i}}$ (powers of prime numbers) and the nodes corresponding to prime numbers between $\frac{N}{3}$ and $\frac{N}{2}$.	
\end{theorem}

\begin{proof}
From the definition of the network $R(N)$, we have that the numbers $n = p_{i}^{\alpha_{i}}$ with $\alpha_{i} > 1$ are only related to one node because they are composite numbers with only one distinct prime factor. On the other hand, a composite number with more than one distinct prime factor has more than one edge. Additionally, a prime number $p$ between $\frac{N}{3}$ and $\frac{N}{2}$ is linked to the composite number $2p$, but is not connected to any other composite number because  $mp > N$ with $m \geq 3$. Finally, any other prime number $p$ in $R(N)$ is related to  $2p$ and $3p$ at least. This completes the proof.
\end{proof}

\begin{theorem}\label{TheoRelevantThisNetwork}
The relevant nodes provided by Algorithm \ref{TeoAlgRelNodeClass} are all prime numbers in the network, i.e prime numbers less than or equal to $\frac{N}{2}$.
\end{theorem}

\begin{proof}
Due to the fact that nodes correspond to natural numbers, we are going to assume that when two nodes have the same degree, the one corresponding to a lower natural number is first in the list of degrees.

For any prime number $p$, there are exactly $\left[\dfrac{N}{p} \right] - 1$ composite numbers less than or equal to $N$ that are divisible by $p$. That means the lowest prime in the list has the highest number of prime factors. Additionally, for composite numbers, the degree is less than $\left[ \mathrm{log}_{2}N\right]$. Therefore, when $\left[\dfrac{N}{2} \right] - 1 \geq \left[ \mathrm{log}_{2}N\right]$, we have that 2 is the first node in the list. The function $f(x) = \left[\dfrac{x}{2} \right] - 1 - \left[ \mathrm{log}_{2}x\right]$ is a non decreasing function of $x$. Therefore, for $N\geq 6$, this method proves to us that 2 is the node of highest degree in the network. When we erase 2 and repeat the process, we have that the remaining lowest prime is 3 and the degree of the remaining composite number nodes in the new network is less than $\left[ \mathrm{log}_{3}N\right]$. Therefore, when $\left[\dfrac{N}{3} \right] - 1 \geq \left[ \mathrm{log}_{3}N\right]$,  we have that 3 is the first node in the new list. Therefore, for $N\geq 6$, this method proves to us that 3 is the node of highest degree in the network.\\
	We will assume now that for an $N$ big enough, all the prime numbers lower than $p$ have been erased. Then the degree of $p$ is still $\left[\dfrac{N}{p} \right] - 1$ and the degree of the remaining composite number nodes in the new network is less than $\left[ \mathrm{log}_{p}N\right]$. Therefore when $\left[\dfrac{N}{p} \right] - 1 \geq \left[ \mathrm{log}_{p}N\right]$  we have than $p$ is the first node in the new list. We should take into account that a composite number $m$ remains connected to the new network after removing the prime numbers lower than $p$ and disconnected nodes, which means that $m$ has $p$ or a prime factor greater than $p$ in its prime factor decomposition and therefore $m > p$. This means that, if $p$ and $m$ have the same degree, then from our first assumption the node corresponding to $p$ is first in the list.
	
Now we are going to analyze the equivalent inequality,
	\begin{equation*}
		\left[\dfrac{N}{p} \right] - 1 - \left[ \mathrm{log}_{p}N\right]\geq 0.
	\end{equation*}
	As before, the function $f(x) = \left[\dfrac{x}{p} \right] - 1 - \left[ \mathrm{log}_{p}x\right]$ is a non decreasing function of $x$. Hence, when $N \geq 2p$ (or equivalently $p \leq \dfrac{N}{2}$) we obtain
	\begin{equation*}
		\left[\dfrac{N}{p} \right] - 1 - \left[ \mathrm{log}_{p}N\right]\geq f(2p) = 0,
	\end{equation*}    
here we are assuming $p>2$.

Thus, the lowest prime number remaining connected in the new network is always the first node in the list of degrees. Conversely, when all prime numbers have been erased, then there is no node still connected to the network. Then, the list of relevant nodes is the list of prime numbers less than or equal to $\dfrac{N}{2}$. This completes the proof.  	 
\end{proof}

\paragraph{Minimal Diameter}
The result from Theorem \ref{TheoDiam5} leads to the following question: Does there exist a minimal constant diameter regardless of the increasing number of nodes?\

In this subsection, we are going to examine this question. Indeed, from the proof of Theorem \ref{TheoConectNet} we know that the diameter of $R^{'}(N)$ is at most four. Moreover, due to the fact that compound numbers are connected only to prime numbers and vice versa, case 2 of the above-referred proof actually shows that the minimal path between two compound numbers that are not relatively prime is exactly four. Hence we have proved the following theorem.
\begin{theorem}
The diameter of the network $R^{'}(N)$ is 4.
\end{theorem}     
Note that we have proved more, namely that the networks $R(N)$ and $R^{'}(N)$ are of small-world type; see \cite{ChanDas, Cor} and the references given there. 
\subsection{Additional Topological Facts}

\paragraph{New Approach to Determine Relevant Nodes}
\noindent

Theorem \ref{TheoRelevantThisNetwork} tells us that the list of relevant nodes will be exactly the list of primes less than or equal to $\frac{N}{2}$. This tells us that, in effect, this criterion is not entirely correct for this specific network, since among these primes there are some with very few connections; for example, those between $\frac{N}{3}$ and $\frac{N}{2}$ have only one connection, and therefore should not be considered relevant nodes. Intuitively, we can infer that the relevant nodes in our network should be formed by those prime numbers between 2 and $\sqrt{N}$, since multiplying these nodes together gives us composite numbers that fall within the range less than $N$. Therefore, we will have many more links associated with these nodes, compared to the primes between $\sqrt{N}$ and $\frac{N}{2}$, which can only be multiplied by primes less than $\sqrt{N}$ to produce composite numbers within the range of composites less than $N$. As a result, these primes between $\sqrt{N}$ and $\frac{N}{2}$ will have far fewer connections than the primes between 2 and $\sqrt{N}$. To correct this problem, we developed two algorithms that aim to improve the bound for finding relevant nodes for networks in general.\\
We call the first of these methods Weighted Degree Difference, and the idea is to divide, across the whole list, the degree of each node by the average node degree of the entire network. This new weighted degree value tells us how many times greater a node's degree is than the network average — in other words, how many times better than the network average that node is.\\
What we would then do is look at the difference between consecutive values in the list. If the difference between these weighted degrees is greater than one, that tells us the node is much more important than the one that follows it, and it would therefore be added to the list of main, or relevant, nodes of the network. On the other hand, if the difference is not greater than one, we would continue examining consecutive nodes in this way, up to a certain number of times, until this condition does or does not occur — producing, at the end, the list of relevant nodes for our network. The algorithm is described as follows.
\begin{algorithm}\label{AlgoWeighDeg}Relevant Nodes Algorithm (Weighted Degree Difference)\\
Let $g(j)$ be the degree of node $j$, $AD$ the average degree of the network, and $g_{p}(j) = \dfrac{g(j)}{AD}$.
	\begin{enumerate}
		\item[(i)] The nodes are sorted from highest to lowest based on their degree.
		\item[(ii)]Starting from the beginning of the list of degrees, the $j$ and $j+1$ nodes are selected; if $g_{p}(j) -g_{p}(j+1) > 1$, then the $j$ node is added to the list of relevant nodes.  
		\item[(iii)] If $g_{p}(j) - g_{p} (j+1) \leq 1$, we examine $g_{p}(j+1) - g_{p}(j+2)$. In general, if any of the differences $g_{p}(j) - g_{p}(j+1) > 1$ for some $j$ in an interval of length $\max(\ln N, L)$, where $ L$ is the length of the updated list of relevant nodes, then node $j$ and the nodes to its left are added to the list of relevant nodes and deleted from the list of degrees.
		\item[(iv)] Steps (ii) and (iii) are repeated, starting with the next node in the list; the differences $g_{p}(j) - g_{p}(j+1) > 1$ are checked for some $j$ in an interval of length $\max(\ln N, L)$, where $L$ is the length of the updated list of relevant nodes. 
		\item[(v)] Stop the process when $g_{p}(j) - g_{p}(j+1) \leq 1$ for all $j$ in an interval of length $\max(\ln N, L)$ or the end of the list of degrees has been reached.
	\end{enumerate}		
\end{algorithm}
An important advantage this algorithm gives us is in terms of computational complexity, since in the worst case we would only need to traverse the list once. This would give us a complexity of $\mathcal{O}(N)$, which does not increase the complexity of the algorithm needed to build the list itself, and therefore would substantially improve the computational cost of the algorithm stated in Algorithm \ref{TeoAlgRelNodeClass}. The computational complexity of Algorithm \ref{AlgoWeighDeg} is  $\mathcal{O}(N \log N)$, which is the time complexity to sort nodes by their degree.

At first glance, the usefulness of the proposed algorithm for identifying relevant nodes is evident from two fundamental aspects. The first is the use of the weighted degree of a node, which measures how many times its degree exceeds the average degree of the nodes in the network. The second is that this metric is evaluated within a neighborhood of the node in the ranked list of highest-degree nodes, over an interval of prescribed length.

Furthermore, in order to assess the effectiveness of the proposed algorithm, we implemented it and applied it to several networks for various values of $N$. Specifically, for the first four cases, the values of $N$ were taken in increments of 250, and thereafter in increments of 1000, up to $N = 64000$.

First, we observed that for values of $N$ less than 4000, the list of relevant nodes produced by this algorithm usually coincides with the list of prime numbers less than or equal to $\sqrt{N}$, that is, the ideal list in our case. For larger values of $N$, however, there were always relevant nodes greater than $\sqrt{N}$. The length of the list of these extra nodes is not an increasing function of $N$; it exhibits a zigzag behavior over the analyzed values. To continue our analysis, we computed several normalized quantities for the values of $N$ mentioned above.

First, we calculated the ratio between the length of the list of relevant nodes and the length of the list of prime numbers less than or equal to $N/2$, which is precisely the list of relevant nodes produced by Algorithm \ref{TeoAlgRelNodeClass}. Among the studied values of $N$, the maximum value of this ratio was 0.2, the minimum was 0.0192, and the average was 0.0401. This shows that the list produced by Algorithm \ref{AlgoWeighDeg} is significantly smaller than the one produced by Algorithm \ref{TeoAlgRelNodeClass}, even though it does not coincide with the list of prime numbers less than or equal to $\sqrt{N}$.

Second, we analyzed the ratio between the length of the list of relevant nodes greater than $\sqrt{N}$ and the length of the list of prime numbers less than or equal to $\sqrt{N}$. The maximum value of this ratio among the studied values of $N$ was 0.5, the minimum was 0, and the average was 0.2632. This indicates that the list produced by the algorithm is not much larger than the ideal one.

Finally, we examined the weighted degree of the last relevant node produced by Algorithm \ref{AlgoWeighDeg}. Among the studied values of $N$, its maximum value was 37.3791, its minimum was 5.0222, and its average was 23.6247. This shows that every node in the list has a degree many times greater than the average degree of the network. In contrast, the last node of the list produced by Algorithm \ref{TeoAlgRelNodeClass} has only one link in almost all cases, the exceptions being $N=9$ and $N=21$, as shown in Theorem \ref{TheoRelevantThisNetwork} and Lemma \ref{LemPriN3N2}.

The second method constitutes a refinement of the method proposed in Algorithm \ref{TeoAlgRelNodeClass}. We call it the No Disconnection Algorithm. Basically, we take the list of relevant nodes produced by Theorem 1, and what we're going to do is clean out from it those that are not actually relevant according to the criterion. To do this, we start from the back of the list and move forward, removing the supposedly relevant node in question. If removing this node does not disconnect any other node, then we remove it from the list of relevant nodes. We then repeat the process with the next node.

\begin{algorithm}(Relevant Nodes Algorithm (No Disconnection))
\begin{enumerate}
\item[(i)] List the relevant nodes obtained through the disconnection algorithm.
\item[(ii)] The list is sorted from smallest to largest based on the node's order in the original network.
\item[(iii)] The node with the lowest order in the list is removed from the network; if no other node disconnects from the network, this node is removed from the list; otherwise, it is retained.
\item[(iv)] Repeat steps (ii)-(iii) with the new list using the original network if the node being analyzed is retained (moving on to the next node) or a new network without the node if the node has been removed from the list of relevant nodes. 
\item[(v)] Stop the process when all nodes in the original list have been processed.
\end{enumerate}	
\end{algorithm}

Applying this algorithm to our network, we in fact obtain that the refined list of relevant nodes will be exactly the prime numbers less than or equal to $\sqrt{N}$, since all composite numbers connected to primes greater than $\sqrt{N}$ must also be connected to primes less than $\sqrt{N}$.\\
In this sense, if we remove any prime number between $\sqrt{N}$ and $\frac{N}{2}$, no node will become disconnected, because all the composite numbers will also be connected to primes less than $\sqrt{N}$. And this holds even if we remove all at once every prime number between $\sqrt{N}$ and $\frac{N}{2}$.\\
The disadvantage of this algorithm is that the computational cost does not improve, since to begin with, we need to run the entire initial algorithm before we can then refine the list of relevant nodes.

\paragraph{Information Dimension}
The information dimension describes the distribution of information in a system. It is based on Shannon's entropy, which quantifies the uncertainty, diversity, and structure of the information contained in the network. Shannon's entropy is defined as \cite{Shan1948} 

\begin{equation}
\label{ShannonEntropy}
I = \lim_{t \to -1} \frac{d}{dt} \sum_{i=1}^{N} p_i^{-t} = -\sum_{i=1}^{N} p_i \ln(p_i),
\end{equation}
where $p_i$ is the probability of the system's microstates. Where  $\sum_{i=1}^{N} {p_i}=1$.

A generalized measure of Shannon entropy is the Tsallis entropy (also called $q$-entropy) \cite{Tsallis1988, Tsallis2010, Tasallis2019cb} in the field of statistical physics, which was proposed using the $q-$logarithm that is defined by
\begin{equation}\label{shannonentropy}
ln_{q}(p_i)=\frac{p_i^{1-q}-1}{1-q},
\end{equation}
($p_i>0$, $q \in {\mathbb R}, q\neq 1$, $ln_{1}(p_i)=ln (p_i)$).

Tsallis entropy resembles the Shannon entropy and is defined 
\begin{equation}\label{tsallisentropy}
T_{q} = \lim_{t \rightarrow-1} D^{t}_{q}\sum_{i=1}^{N}{p_{i}^{-t}}=-\sum^{N}_{i=1}p_i ln_q {p_i}.
\end{equation}
where $D^{t}_{q}$ of a function $f$ given by $D^{t}_{q}f(t)=\displaystyle\frac{f(qt)-f(t)}{(q-1)t}$, $t\neq 0$, stands for the Jackson \cite{Jackson} $q-$derivative, see \cite{Abe, Johal, LS} for in depth details. It is notably seen that Tsallis entropy is a one-parameter generalization of Shannon entropy. This means that Tsallis entropy converges to it as $q \rightarrow 1$. The entropic index $q$ is correlated with the complex internal interactions of the system, and usually, it is estimated empirically.

These two entropies can be computed on complex networks by first obtaining the box-covering of the networks \cite{Song2007} using boxes of size $\varepsilon=[1,\Delta+1]$, where $\Delta$ is the diameter of the networks. The diameter means the largest shortest path from any node to another one. Second, we compute each $p_i=\frac{n}{N}$ for $i=1,2,..., N_b(\varepsilon)$, where $n$ is the number of nodes in a given box obtained by the box-covering algorithm \cite{Song2007}, $N_b(\varepsilon)$ is the number of boxes of a given diameter $\varepsilon$, and $N$ is the total number of nodes in the network.

The entropic index $q$ of Tsallis entropy is estimated by \cite{Ramirez2023,RAMIREZARELLANO2020}

\begin{equation}
\label{qepsilon}
q_\varepsilon = \frac{(\Delta - 1) \cdot N_b(\varepsilon)}{\sum_{\varepsilon = 1}^{\Delta+1} N_b(\varepsilon)},
\end{equation} 
where $\Delta$ and $N_b(\varepsilon)$ have been previously described.

Let us consider a generic entropy $I(\varepsilon)$ in the case of a discrete uniform distribution; a definition of the information dimension  is introduced in \cite{Dai} as follows
\begin{equation}
\label{informationdimension}
d_{I}=-\lim_{\varepsilon \to 0}\frac{I(\varepsilon)}{\ln \varepsilon} 
\end{equation}
where $\varepsilon$ is the size of the boxes to cover the network. Since the size of a box in complex networks is $\varepsilon \geq 1$, here and below the expression $\lim {\varepsilon \to 0}$ lacks sense, but it serves to keep up the standard notation.  

From Eq. \ref{informationdimension} we can assert that
\begin{equation}\label{idf}
I(\varepsilon) \sim -d_{I} \ln \varepsilon + \beta,
\end{equation} 
for some constant $\beta$, where $\varepsilon$ the size of the boxes to cover the network. 

Now consider that $I(\varepsilon)$ is the Shannon entropy (Eq. \ref{shannonentropy}); it implies computing the entropy for $\varepsilon=[1,\Delta+1]$, then the information dimension is estimated by non-linear regression using Eq. (\ref{idf}). Similarly, the Tsallis information dimension can be computed.

An example of the computation of Shannon and Tsallis entropies of a network is described below. Figure \ref{fig:Covering} shows the covering of a network of prime factors and composite numbers with $n=10$. The not-connected nodes and self-loops were deleted. The covering of the network with boxes of size $\varepsilon=4$ is shown in Figure \ref{fig:Covering} a), for further details on the box-covering algorithm, see \cite{Song2007}. The color of each node corresponds to the box it belongs to. Hence, $B_1=\{2,3,4,6,8\}$ , $B_2=\{5\}$, $B_3=\{9\}$, $B_4=\{10\}$ and $N_b(\varepsilon=4)=4$. Similarly, for $\varepsilon=5$, the covering is introduced in Figure \ref{fig:Covering} b). The boxes obtained are: $B_1=\{2,3,4,5,6,8,10\}$ , $B_2=\{9\}$ so $N_b(\varepsilon=5)=2$. Then the probabilities for $\varepsilon=4$ can be easily computed from the boxes above such $p_1=\frac{5}{8}$, $p_2=\frac{1}{8}$, $p_3=\frac{1}{8}$ and $p_4=\frac{1}{8}$. The Shannon entropy (Eq. \ref{shannonentropy}) for $\varepsilon=4$; $I(\varepsilon=4)=-[\frac{5}{8} \ln(\frac{5}{8}) +\frac{1}{8} \ln(\frac{1}{8})+\frac{1}{8} \ln(\frac{1}{8})+\frac{1}{8} \ln(\frac{1}{8})]=1.073$. The remaining Shannon entropy values are shown in Table \ref{tab:information}.

\begin{figure}[H] 
	\centering
		\includegraphics[scale=0.35]{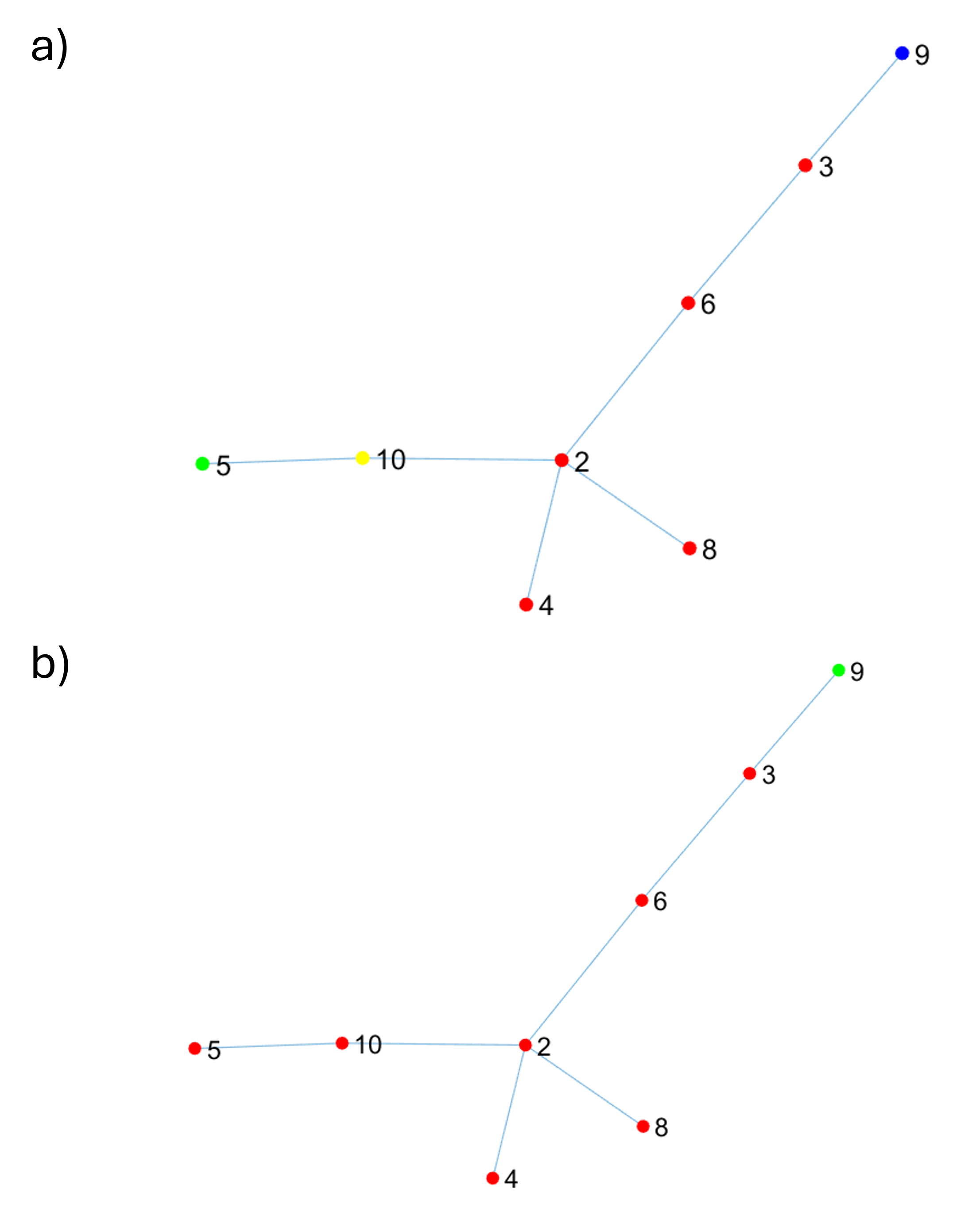}
	\caption{A network covering built with $ n=10$, excluding unconnected nodes and self-loops. The covering with boxes of size a)$\varepsilon=4$ and b)$\varepsilon=5$. }
	\label{fig:Covering}
\end{figure}

On the other hand, the entropic index using Eq. (\ref{qepsilon}) is shown in Table \ref{tab:information} as well as the Tsallis entropy of Eq.(\ref{tsallisentropy}).
\begin{table}[H]
\centering
\begin{tabular}{c c c c c c}   
\toprule
$\varepsilon$ & $N_b(\varepsilon)$ & $I(\varepsilon)$& $q_\varepsilon$ & $T_{q_\varepsilon}$ \\
\midrule
1 & 8 & 2.079 &   1.103 & 1.871\\
2 & 8 & 2.079 &   1.103 & 1.871 \\
3 & 6 & 1.733 &   0.828 & 2.034\\
4 & 4 & 1.074 &   0.552 & 1.615\\
5 & 2 & 0.377 &   0.276 & 0.728\\
6 & 1 & 0 &   0.138 &0 \\
 \bottomrule
\end{tabular}
 \caption{The Shannon $I(\varepsilon)$, and Tsallis ($T_{q_\varepsilon}$) entropy  of the network of Figure \ref{fig:Covering}.}
    \label{tab:information}
\end{table} 
The information dimension ($d_I$) and Tsallis information dimension ($d_T$) are computed using Eq.(\ref{idf}) by replacing $I(\varepsilon)$ with the respective entropy. A non-linear regression is then employed to approximate $d_I$. Figure \ref{fig:Figure2} a) show the Shannon entropy of networks built using $N=[250,500,700]$ and $N=[1000-7000]$ in steps of 1000. Figure \ref{fig:Figure2} b) shows the information dimension of networks constructed using $N=[1000-27000]$ in steps of 1000.

\begin{figure}[htb]
	\centering
		\includegraphics[width=.8\textwidth]{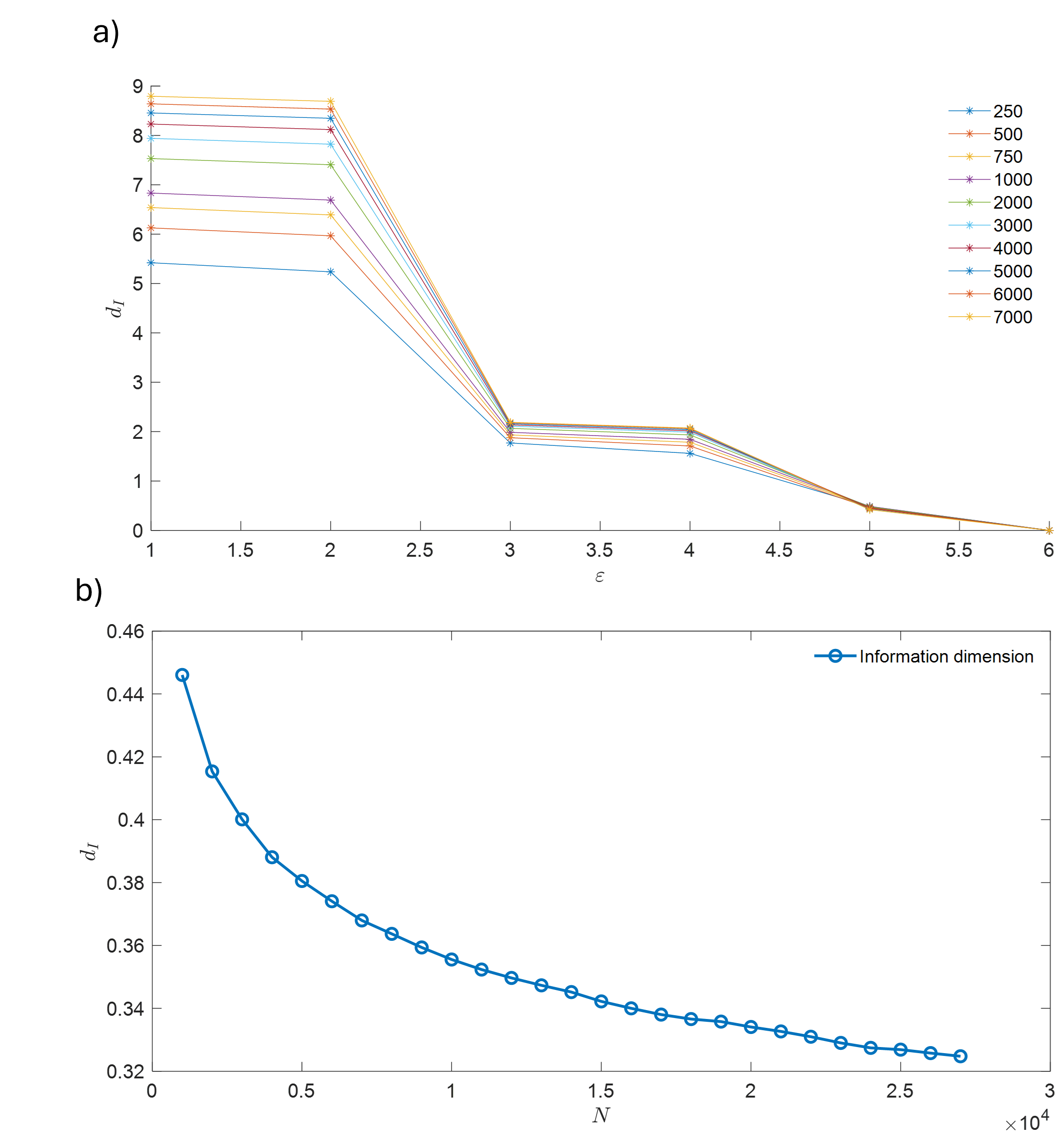}
	\caption{a) The Shannon entropy of several networks built with different $N$ and b) the information dimension varying with $N$}
	\label{fig:Figure2}
\end{figure}

Similarly, the Tsallis entropy ($d_T$) is shown in Figure \ref{fig:Figure3} a) of the networks built using $N=[250,500,700]$ and $N=[1000-7000]$ in steps of 1000  is shown. Also, the Tsallis information dimension ($d_T$) of the networks constructed with $N=[1000-7000]$ in steps of 1000 is shown in Figure \ref{fig:Figure3} b). 

\begin{figure}[htb]
	\centering
		\includegraphics[width=.8\textwidth]{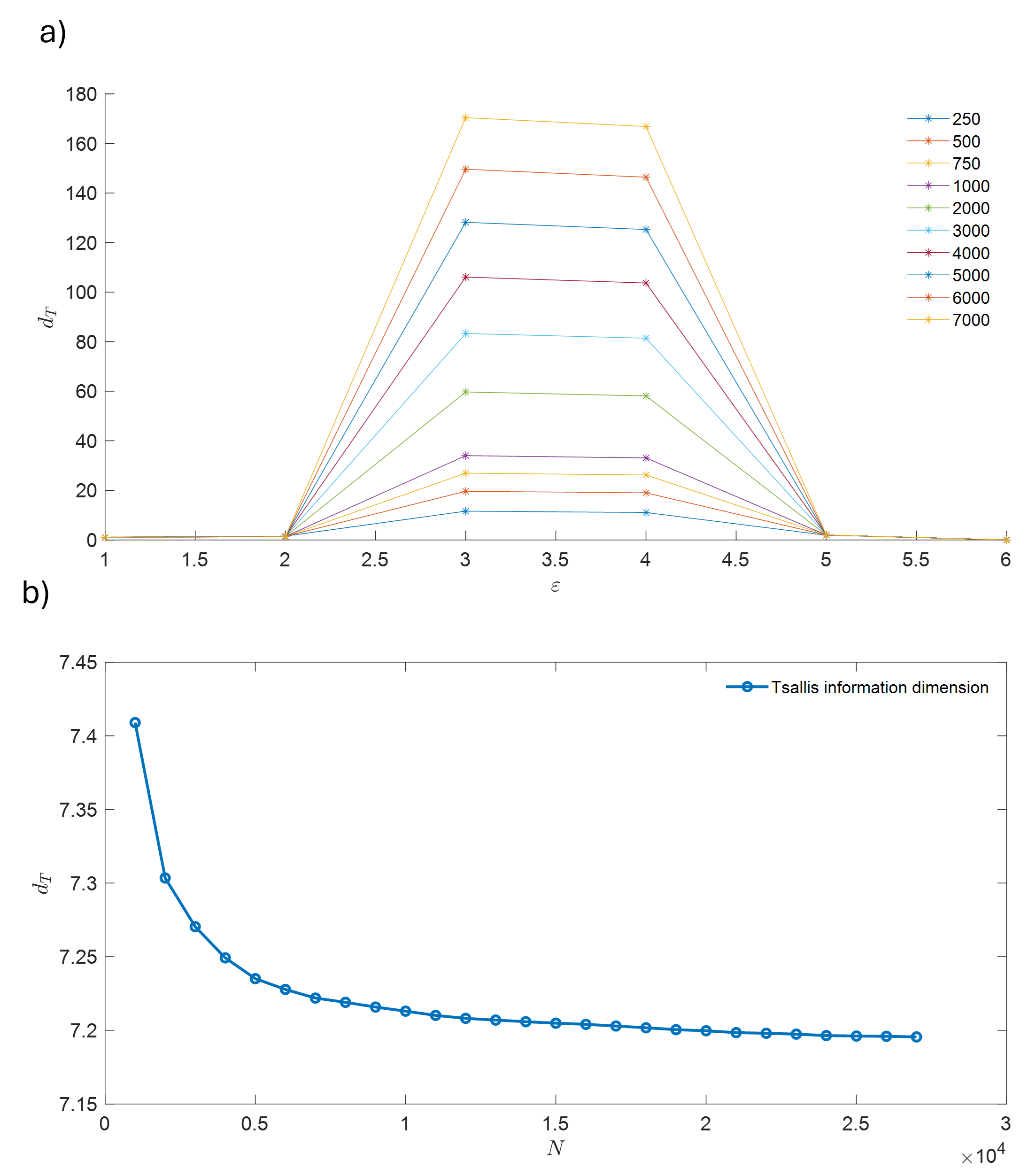}
	\caption{a) The Tsallis entropy of several networks built with different $N$ and b) the Tsallis information dimension varying $N$}
	\label{fig:Figure3}
\end{figure}

The value of the entropies increases as $N$ increases; see Figures \ref{fig:Figure2} a) and \ref{fig:Figure3} a). Thus, the information dimension ($d_I$) and Tsallis information dimension ($d_T$) decrease, as can be shown in Figures \ref{fig:Figure2} b) and \ref{fig:Figure3} b). In contrast, the pattern in Shannon entropy differs from that of Tsallis entropy, which reaches maximum values in $\varepsilon=[3,4]$; see Figures \ref{fig:Figure2} a) and \ref{fig:Figure3} a).

5We wish to acknowledge the support of the author community in using
5REV\TeX{}, offering suggestions and encouragement, testing new versions,

\section{Concluding remarks}
Prime numbers have been the subject of intense human interest from the earliest days of mathematics development, which is possibly a consequence of being unpredictably and irregularly distributed among the other integers. As a consequence of this continuing interest, a large number of concepts and methods, especially in the area of Number Theory, have been developed and systematically applied.

The present work characterizes topological features of a bipartite complex network of prime factors and composite numbers. Interesting results have been obtained and discussed, in particular those listed as follows:
\begin{enumerate}
\item[(i)] The diameter of the network is exactly five with disregard of the amount of nodes. 
\item[(ii)] Certain nontrivial sub-networks have minimal diameter equal to four, which is independent of the number of nodes. The last shows that this network possesses the small-world property. 
\item[(iii)] It characterizes exactly the number of boundary nodes in the network as well as presents a couple of algorithms to find out the most important, relevant, influential, or central nodes of the network.
\item[(iv)] Computing the Shannon and Tsallis information dimensions shows that each characterizes network topology differently, and both decrease as the number of nodes increases. Furthermore, both entropies indicate that box coverings, with diameters $\varepsilon=[3,4]$, group the nodes into approximately the same number of boxes independently of the number of nodes used to build the network. 
\end{enumerate}

\section*{Declarations}

\paragraph{Acknowledgments:}
This article was written during a postdoctoral scientific stay of C. D. Tamayo Castro at Instituto Polit\'ecnico Nacional, kind hospitality is gratefully acknowledged. The author wishes to thank the Secretar\'ia de Ciencia Humanidades Tecnolog\'ia e Innovaci\'on (SECIHTI) for the financial support (Grant Number 957110).
\paragraph{Conflicts of interest:}
The authors declare that they have no competing interests regarding the publication of this paper.
\paragraph{Author Contributions.} 
All authors contributed to the study conception and provided critical feedback and helped shape the research, analysis, and manuscript.


\end{document}